\documentclass[letterpaper, 10 pt, conference]{IEEEtran}
\IEEEoverridecommandlockouts

\usepackage{fullpage, booktabs, multirow, graphicx, subcaption, url, enumitem, cite, float}
\usepackage[usenames,dvipsnames]{xcolor}
\usepackage{amsmath,amssymb,amsthm,mathrsfs,bm,pifont}
\usepackage{scalerel,stackengine}
\usepackage{algpseudocode, algorithm}
\usepackage[subtle, title]{savetrees}
\newcommand{\ssmargin}[2]{{\color{blue}#1}{\marginpar{\color{blue}\raggedright\scriptsize [SS] #2 \par}}}

\newcommand{\thmargin}[2]{{\color{green}#1}{\marginpar{\color{green}\raggedright\scriptsize [TH] #2 \par}}}

\newcommand{\X}{\mathcal{X}}
\newcommand{\U}{\mathcal{U}}
\newcommand{\Pp}{\mathcal{P}}
\newcommand{\reals}{\mathbb{R}}
\newcommand{\Prob}{\mathcal{P}}
\newcommand{\argmin}{\operatornamewithlimits{\mathrm{argmin}}}
\newcommand{\ulb}{\underline{u}}
\newcommand{\uub}{\overline{u}}
\newcommand{\CR}{\mathcal{CR}}

\newcommand{\J}{\mathcal{J}}
\newcommand{\ol}{\texttt{\bf{OL}}}
\newcommand{\cl}{\texttt{\bf{CL}}}
\newcommand{\clg}{\texttt{\bf{CL}}_{\gamma}}

\theoremstyle{nospace} \newtheorem{theorem}{Theorem}
\theoremstyle{nospace} 
\theoremstyle{nospace} 
\theoremstyle{nospace} 
\theoremstyle{nospace} \newtheorem{remark}{Remark}
\theoremstyle{nospace} 
\theoremstyle{nospace} 
\theoremstyle{nospace} 
\theoremstyle{nospace} 

\makeatletter
\newcommand{\linebreakand}{%
\end{@IEEEauthorhalign}
\hfill\mbox{}\par
\mbox{}\hfill\begin{@IEEEauthorhalign}
}
\makeatother

\graphicspath{{./fig/}}

\title{\LARGE \bf Optimizing Trajectories with Closed-Loop Dynamic SQP}
\date{}

\author{Sumeet Singh$^{1}$, Jean-Jacques Slotine$^{1}$, Vikas Sindhwani$^{1}$
    \thanks{$^{1}$ Sumeet Singh, Jean-Jacques Slotine (MIT), and Vikas Sindhwani are with Robotics at Google,
   	 New York City, NY 10011, USA
   	 {\tt\small \{ssumeet,jeanjacquess, sindhwani\}@google.com}}%
}

\begin{document}
    \maketitle
    \begin{abstract}
    Indirect trajectory optimization methods such as Differential Dynamic Programming (DDP) have found considerable success when only planning under dynamic feasibility constraints. Meanwhile, nonlinear programming (NLP) has been the state-of-the-art approach when faced with additional constraints (e.g., control bounds, obstacle avoidance). However, a na{\"i}ve implementation of NLP algorithms, e.g., shooting-based sequential quadratic programming (SQP), may suffer from slow convergence -- caused from natural instabilities of the underlying system manifesting as poor numerical stability within the optimization. Re-interpreting the DDP closed-loop rollout policy as a \emph{sensitivity-based correction to a second-order search direction}, we demonstrate how to compute analogous closed-loop policies (i.e., feedback gains) for \emph{constrained} problems. Our key theoretical result introduces a novel dynamic programming-based constraint-set recursion that augments the canonical ``cost-to-go" backward pass. On the algorithmic front, we develop a hybrid-SQP algorithm incorporating DDP-style closed-loop rollouts, enabled via efficient \emph{parallelized} computation of the feedback gains. Finally, we validate our theoretical and algorithmic contributions on a set of increasingly challenging benchmarks, demonstrating significant improvements in convergence speed over standard open-loop SQP.
    \end{abstract}


\section{Introduction}

Trajectory optimization forms the backbone of model-based optimal control with myriad applications in robot mobility and manipulation \cite{SchulmanEtAl2014,KalakrishnanChittaEtAl2011,SindhwaniRoelofs2017,HowellJackson2019,TassaErez2007}. The problem formulation is as follows: consider a robotic system with state $x \in \reals^n$, control input $u \in \reals^m$, subject to the discrete-time dynamics:
\begin{equation}
    x_{k+1} = f(x_k, u_k), \quad k \in \mathbb{N}_{\geq 0}.
\label{dyn}
\end{equation}
Let $N \in \mathbb{N}_{>0}$ be some fixed planning horizon. Given some initial state $x_0$, the trajectory optimization problem is as follows:
\begin{subequations}
\begin{align}
    \min_{\bm{u}, \bm{x}} \qquad &\sum_{k=0}^{N-1} l_k(x_k, u_k) + l_N(x_N) =: \J(\bm{u}, \bm{x})\\
    \mathrm{s.t.} \qquad & k=0,\ldots,N-1: \quad \begin{cases} x_{k+1} = f(x_k, u_k) \\
    c_k^u(u_k) \geq 0 \\ c_{k+1}^x(x_{k+1}) \geq 0 \end{cases},
\end{align}
\label{trajopt}
\end{subequations}
where we use $(\bm{u}, \bm{x})$ to denote the concatenations $(u_0, \ldots, u_{N-1})$ and $(x_0, \ldots, x_{N})$, respectively. Here, $l_k : \reals^n \times \reals^m \rightarrow \reals_{\geq 0}$ is the running cost, $l_N: \reals^n \rightarrow \reals_{\geq 0}$ is the terminal cost, and $c_k^x : \reals^n \rightarrow \reals^{n_x}, c_k^u: \reals^m \rightarrow \reals^{n_u}$ are the vector-valued constraint functions on the state and control input. We assume that the control constraint encodes simple box constraints: $\ulb \leq u_k \leq \uub$, though, the results in this paper may be generalized beyond this assumption.

Solution methods generally fall into one of two approaches: optimal control-based (indirect methods), or optimization-based (direct methods). The former leverages necessary conditions of optimality for optimal control, such as dynamic programming (DP), while the latter treats the problem as a pure mathematical optimization program~\cite{Bertsekas2016}. A further sub-categorization of the direct method distinguishes between a \emph{Full} or a \emph{Condensed} formulation, where the former treats both the states and controls as optimization variables, subject to dynamics equality constraints, while the latter optimizes only over the control variables, with the dynamics implicit.

\begin{figure}[t]
  \centering
  \includegraphics[width=0.5\textwidth]{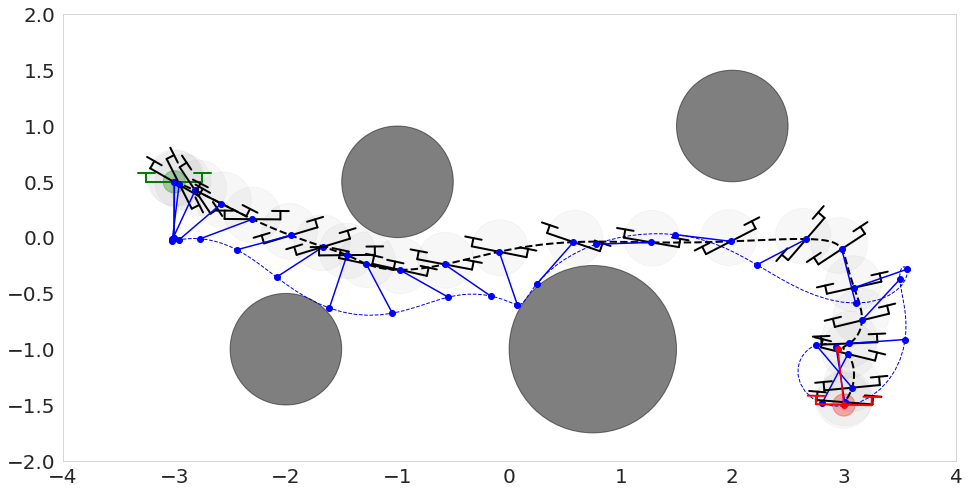}  
\caption{\footnotesize{{\it Motion planning using our methods for a planar quadrotor with attached pendulum, starting at rest with the pendulum in stable equilibrium (top-left, green), and ending at rest with the pendulum upright (bottom-right, red).}}}
\label{fig:quadpole2}\vspace{-0.5cm}
\end{figure}
Lacking constraints beyond dynamic feasibility, ubiquitous indirect methods~\cite{LiTodorov2004,Mayne1973,MurrayYakowitz1984,DunnBertsekas1989} such as Differential Dynamic Programming (DDP) and its Gauss-Newton relaxation, iterative Linear Quadratic Regulator (iLQR) rely upon the DP recursion to split the full-horizon planning problem into a sequence of one-step optimizations, and alternate between a backward and forward pass through the time-steps. The backward pass recursively forms quadratic expansions of the optimal cost-to-go function and computes a time-varying affine perturbation policy that is subsequently rolled out through the system's dynamics in the forward pass to yield the updated trajectory iterate. Under mild assumptions, DDP locally achieves quadratic convergence, and the proof relies upon establishing the close link between DDP and the Newton method, as applied to the \emph{condensed} optimization-based formulation~\cite{MurrayYakowitz1984}. The stability properties of the underlying nonlinear system  manifest as numerical stability during the optimization process, and hence the \emph{closed-loop} nature of the forward pass in DDP typically leads to better performance~\cite{LiaoSchoemaker1992} than Newton's method, which implements \emph{open-loop} rollouts.

Computing DDP-style  closed-loop updates within the \emph{constrained} setting (beyond dynamics feasibility) is much more challenging, since quadratization of the cost and linearization of the dynamics and constraints yields \emph{constrained} quadratic programs (QPs) with \emph{piecewise-affine} optimal perturbation policies~\cite{BorrelliBemporad2017}, and complexity growing exponentially in the number of constraints and horizon of the problem. Consequently, direct methods (featuring open-loop updates) are the prevailing solution approach, typically combined with interior-point or SQP algorithms. Aside from possessing more variables, the direct formulation must additionally resolve dynamic feasibility, which can be non-trivial and lead to slower convergence even for unconstrained systems, as compared to indirect methods. Moreover, while the condensed formulation yields smaller problems, instabilities have been observed~\cite{TennyWright2004} due to divergences between the predicted step from the linearized dynamics, and the open-loop nonlinear rollout, leading to vanishing step-sizes.


\noindent {\bf Contributions:} Re-interpreting the canonical DDP closed-loop rollout as a \emph{sensitivity-based correction to a second-order search direction}, we demonstrate how to compute a locally affine approximation to the constrained perturbation policies, i.e., a set of feedback gains similar to those employed by the DDP rollouts. Our key theoretical result states that in the constrained setting, one must first compute an optimal perturbation sequence about the current trajectory iterate by solving a full-horizon QP (as opposed to the one-step DP backward recursion), and then augment the canonical cost-to-go backward pass with a constraint-set recursion. We then demonstrate how to approximate the desired feedback gains using an efficient, parallelized algorithm, eliminating the backward pass. The closed-loop rollout is integrated into an SQP line-search, yielding a hybrid indirect/direct algorithm that combines the theoretical foundations of SQP for constrained optimization with the algorithmic efficiencies of DDP-style forward rollouts. The method is rigorously evaluated within several environments, where we confirm significant convergence speed improvements over na{\"i}ve (i.e., open-loop) SQP.

\noindent {\bf Related Work:} Quasi-DDP methods for constrained trajectory optimization fall into one of three main categories: control-bounds only~\cite{Tassa2014},\cite{MartiSolaEtAl2020},  modified backward pass via KKT analysis~\cite{GiftthalerBuchli2017,AoyamaBoutselis2020,XieLiu2017,Yakowitz1986,LantoineRussell2012}, and augmented Lagrangian methods~\cite{PlancherManchester2017,HowellJackson2019,SindhwaniRoelofs2017,MaCheng2020,PavlovShames2021}. We provide a comprehensive overview of these approaches in Appendix~\ref{app:related}\footnote{All appendices referenced herein may be found in the online version of this work~\cite{SinghSlotine2021}.}.

\section{Shooting SQP}
\label{sec:shooting}

We detail below the core algorithmic steps for SQP, as applied to the shooting formulation of problem~\eqref{trajopt}, i.e., where dynamics are treated implicitly and we optimize only over the control sequence $\bm{u}$. The three steps are~\cite{GillMurray1986,GillMurray2005}: (i) solving a QP sub-problem to compute a search direction, i.e., a sequence of control perturbations $\bm{\delta u} = (\delta u_0, \ldots, \delta u_{N-1})$, (ii) performing line-search along $\bm{\delta u}$ using a merit function, and (iii) monitoring termination conditions. We provide some details regarding (i) and (ii) here, and refer the reader to Appendices~\ref{app:sqp} and~\ref{app:ocp} for the rest.

Let $c_k(x_k, u_k)$ denote the concatenation $(c_k^x(x_k), c_k^u(u_k))$, $k=0,\ldots, N$ (where $c_N^u$ is null), and let $y_k$ denote the corresponding dual variable. Consider the Lagrangian at current primal-dual iterate $(\bm{u}, \bm{y})$:
\begin{equation}
    \mathcal{L}(\bm{u}, \bm{x}, \bm{y}) =  \sum_{k=0}^{N-1} l_k - y_k^T c_k + l_N - y_N ^T c_N
\label{trajopt_Lag_implicit}
\end{equation}
For brevity, we omit the explicit arguments $(x_k, u_k)$ where possible. As we are only optimizing over $\bm{u}$, let $\bm{x}[\bm{u}]$ represent the state trajectory starting at $x_0$ obtained from propagating the open-loop control sequence $\bm{u}$ through the discrete-time dynamics in~\eqref{dyn} for $k=0,\ldots,N-1$. Define the ``reduced" objective and Lagrangian as $\J_R(\bm{u}) = \J(\bm{u}, \bm{x}[\bm{u}])$ and $\mathcal{L}_R(\bm{u}, \bm{y}) := \mathcal{L}(\bm{u}, \bm{x}[\bm{u}], \bm{y})$, respectively, and define the sets $\{\delta \X_k\}_{k=1}^{N}$ and $\{\delta \U_k\}_{k=0}^{N-1}$ as $\delta \X_k := \{\delta x : c_k^x + J_k^x \delta x \geq 0\}$ and $\delta \U_k := \{ \delta u : c_k^u + J_k^u \delta u \geq 0\}$. Then, the QP sub-problem takes the form:
\begin{subequations}
\begin{align}
        \min_{\bm{\delta u}, \bm{\delta x}} \quad &\langle\bm{\delta u}, \nabla_{\bm{u}} \J_R(\bm{u})\rangle + \dfrac{1}{2} \langle \bm{\delta u}, \nabla_{\bm{u}}^2 \mathcal{L}_R(\bm{u}, \bm{y}) \bm{\delta u} \rangle  \label{qp_obj_implicit}\\
        \mathrm{s.t.} \quad &\delta x_{k+1} = A_k \delta x_k + B_k \delta u_k,\  k = 0,\ldots, N-1 \label{lin_dyn} \\
        \quad & \delta u_k \in \delta\U_k,\quad \delta x_{k+1} \in \delta\X_{k+1} , \ k=0,\ldots,N-1 \label{con_lin},
\end{align}
\label{qp_prob_implicit}
\end{subequations}
where $\delta x_0 = 0$ and $\langle \cdot, \cdot \rangle$ denotes the standard Euclidean dot product; $(A_k, B_k)$ are the dynamics Jacobians $(\partial f/\partial x, \partial f/\partial u)(x_k, u_k)$, and $(J_k^u, J_k^x)$ are the constraint Jacobians $(\partial c_k^u(u_k)/\partial u, \partial c_k^x(x_k)/\partial x)$. The sub-problem objective~\eqref{qp_obj_implicit} has the form:
\begin{equation}
\begin{split}
    &\sum_{k=0}^{N-1} \overbrace{\left[q_k^T \delta x_k + r_k^T \delta u_k + \dfrac{1}{2} \begin{bmatrix} \delta x_k \\ \delta u_k \end{bmatrix}^T Z_k \begin{bmatrix} \delta x_k \\ \delta u_k \end{bmatrix}\right]}^{:=\tilde{l}_k(\delta x_k, \delta u_k)} + \\
    &\quad + \underbrace{\dfrac{1}{2} \delta x_N^T Z_N \delta x_N + q_N^T \delta x_N}_{:= \tilde{l}_N(\delta x_N)}
\end{split}
\label{qp_obj_explicit}
\end{equation}
where the terms $\{q_k, r_k, Z_k\}$ are provided in Appendix~\ref{app:ocp}. Let $(\bm{\delta u}^*, \bm{\delta x}^*)$ represent the optimal primal, and $\hat{\bm{y}}$ the optimal inequality dual solutions for~\eqref{qp_prob_implicit}, and define the dual search direction $\bm{\delta y}^* := \hat{\bm{y}} - \bm{y}$.

\subsection{Line-Search}\label{sec:line-implicit}

For $\alpha \in (0, 1]$, define $\bm{u}[\alpha] := \bm{u} + \alpha \bm{\delta u}^*$ and $\bm{x}[\bm{u}[\alpha]] := \bm{x}[\alpha]$. The line-search merit function is defined as:
\begin{equation}
    \phi(\alpha; \rho) = \mathcal{M}_I(\bm{u}[\alpha], \bm{y} + \alpha \bm{\delta y}^*, \bm{s} + \alpha \bm{\delta s^*}; \rho), \  \bm{x} = \bm{x}[\alpha],
\label{line-search}
\end{equation}
where $\mathcal{M}_I$ is the augmented Lagrangian function for problem~\eqref{trajopt}, $\bm{s}$ is a vector of slack variables for the inequality constraints with search direction $\bm{\delta s}^*$, introduced solely for the line-search, and $\rho = \{\rho_k\}_{k=0}^{N}$ is a set of penalty parameters. Please see Appendix~\ref{app:ocp} for details on $\{\bm{s}, \bm{\delta s}, \rho\}$.

\section{Dynamic Programming SQP}
\label{sec:dp-sqp}

Implicit within the line-search is the \emph{open-loop} rollout along the search direction $\bm{\delta u}^*$, i.e., $\bm{x}[\alpha] = \bm{x}[\bm{u} + \alpha \bm{\delta u}^*]$. For unstable nonlinear systems, this state trajectory may differ significantly from $\bm{x} + \bm{\delta x}^*$, the ``predicted" sequence from the QP sub-problem, forcing the line-search to take sub-optimal step-sizes and slowing convergence. This observation is corroborated in~\cite{LiaoSchoemaker1992} in context of comparing DDP and Newton methods for unconstrained problems, and within~\cite{TennyWright2004} in the constrained context. Our objective therefore, is to efficiently compute a set of feedback gains to perform DDP-style \emph{closed-loop} rollouts within the SQP line-search. We hypothesize that such an enhancement will (i) improve the numerical stability of the line-search, and (ii) accelerate convergence of Shooting SQP. We first demonstrate how the classical DP recursion is ill-posed in the context of constrained trajectory optimization, and propose a correction inspired from sensitivity analysis.

\subsection{Sensitivity-Based Dynamic Programming}
The starting point for the derivation of iLQR and DDP algorithms for unconstrained problems is with the Bellman form of the optimal cost-to-go function:
\[\small
    \begin{split}
    V_k(x) &:= \min_{\pi_k} \left[l_k(x, \pi_k(x)) + V_{k+1}(f(x, \pi_k(x))) \right],\  k =0,\ldots, N-1 \\ V_N(x) &:= l_N(x),
    \end{split}
\]
where $\pi_k: \reals^n \rightarrow \reals^m$ is a policy for time-step $k$, mapping states to controls. For a non-optimal state-control sequence $(\bm{x}, \bm{u})$, consider the local expansion of the optimal cost-to-go function:%
\begin{small}
\[
    \begin{split}
    &\delta V_k(\delta x_k) := V_k(x_k + \delta x_k) \\
    &\quad = \min_{\delta \pi_k} \left[\underbrace{l_k(x_k + \delta x_k, u_k + \delta \pi_k(\delta x_k)) + \delta V_{k+1}(\delta x_{k+1}) }_{:= Q_k(\delta x_k, \delta \pi_k)}\right],
    \end{split}
\]
\end{small}%
where $\delta x_{k+1} = f(x_k + \delta x_k, u_k + \delta \pi_k(\delta x_k)) - x_{k+1}$, and $\delta \pi_k: \reals^n \rightarrow \reals^m$ is a \emph{perturbation policy} for time-step $k$ at $x_k$, as a function of $\delta x_k$. Now, by recursively (in a backward pass) taking quadratic approximations of the state-action variation function $Q_k$ about $(\delta x_k, \delta u_k) = (0, 0)$, one can solve for an affine approximation to the minimizing perturbation policy. In particular, let $\breve{Q}_k$ represent the quadratic approximation of $Q_k$, and define $\delta \breve{\pi}_k^*(\delta x_k) := \argmin_{\delta u} \breve{Q}_k(\delta x_k, \delta u)$. Then $\delta \breve{\pi}_k^*(\delta x_k) = \overline{\delta u}_k + K_k \delta x_k$. For step-length $\alpha$, this perturbation policy is rolled out to obtain the new trajectory iterate:
\begin{equation}
    \begin{split}
    \delta x_{k+1}[\alpha] &= f(x_k + \delta x_k[\alpha], u_k + \delta u_k[\alpha]) - x_{k+1}, \\
    \delta u_k[\alpha] &= \alpha \overline{\delta u}_k + K_k \delta x_k[\alpha], \quad \alpha \in (0, 1],
    \end{split}
\label{ddp_roll}
\end{equation}
where $\delta x_0[\alpha] = 0$. Notice that one may interpret the terms of the \emph{unconstrained} perturbation policy $\delta \breve{\pi}_k^*$ as follows:
\begin{equation}
    \overline{\delta u}_k := \delta \breve{\pi}_k^*(0), \quad K_k := \dfrac{\partial \delta \breve{\pi}_k^*(0)}{\partial \delta x_k}.
\label{ddp_terms}
\end{equation}
\begin{remark}\label{rem:pert_uncon}
Since $\delta \breve{\pi}^*_k(\delta x_k)$ is the solution of an unconstrained convex quadratic, the argument $0$ is redundant for the sensitivity matrix $K_k$. This will not be the case in the constrained setting.
\end{remark}
Consider now the constrained setting, and define for $k \in \{0,\ldots, N-1\}$:
\begin{equation}\small
    \delta \pi_k^*(\delta x_k) := \argmin_{\delta u_k \in \delta \U_k} \quad \tilde{l}_k(\delta x_k, \delta u_k) + \delta V_{k+1}(\delta x_{k+1})
\label{pi_pert}
\end{equation}
where $\tilde{l}_k$ is the stage-$k$ term in~\eqref{qp_obj_explicit}, $\delta x_{k+1} = A_k \delta x_k + B_k \delta u_k$, and $\delta V_{k+1}$ is the optimal ``cost-to-go" for problem~\eqref{qp_prob_implicit}. That is, for $k+1 = N$, $\delta x_N \in \delta \X_N$, $\delta V_N(\delta x_N) = \tilde{l}_N(\delta x_N)$, while for $k+1 \in \{1, \ldots, N-1\}$, $\delta x_{k+1} \in \delta \X_{k+1}$, $\delta V_{k+1}(\delta x_{k+1})$ is the optimal value of the \emph{tail-truncation} of QP sub-problem~\eqref{qp_prob_implicit}, starting at time-step $k+1$ at $\delta x_{k+1}$.

 Notice that since $\delta V_{k+1}$ is the optimal value of a \emph{constrained} QP, $\delta \pi_k^*(0)$ and the sensitivity matrix $\partial \delta \pi_k^*(0)/\partial \delta x_k$ (paralleling the terms defined in~\eqref{ddp_terms}) may be \emph{ill-defined}, for instance when the tail sub-problem is infeasible at $\delta x_k = 0$. This is a consequence of the linearized constraints, irrespective of the objective function used to define the DP recursion.
 
 Instead, consider the following equivalent re-arrangement of the \emph{unconstrained} DDP control law in~\eqref{ddp_roll}:
\begin{equation}\small
    \begin{split}
    \delta u_k[\alpha] &= \alpha \overline{\delta u}_k + K_k \delta x_k[\alpha] \\
        &= \alpha \overline{\delta u}_k + K_k \alpha \delta x_k^L + K_k (\delta x_k[\alpha] - \alpha \delta x_k^L) \\
        &= \alpha \delta u_k^L + K_k (\delta x_k[\alpha] - \alpha \delta x_k^L)
    \end{split}
\label{ddp_alt}
\end{equation}
where, the sequence $(\bm{\delta x}^L, \bm{\delta u}^L)$ is defined by the rollout of $\delta \breve{\pi}^*_k$ via the linearized dynamics:
\begin{equation}
    \begin{split}
    \delta x_{k+1}^L &= A_k \delta x_k^L + B_k \delta u_k^L, \quad \delta x_0^L = 0 \\
    \delta u_{k}^L &:= \delta \breve{\pi}^*_k(\delta x_k^L).
    \end{split}
\label{lin_rollout}
\end{equation}
In light of the homogeneity of the above recursion (i.e., the sequence $\alpha \bm{\delta u}^L$ rolled out via the linear dynamics yields $\alpha \bm{\delta x}^L$), eq.~\eqref{ddp_alt} suggests interpreting $\bm{\delta u}^L$ as a \emph{search-direction}, $\alpha \bm{\delta u}^L$ as the search step, and the feedback term as a sensitivity-based correction. Thus, we may interpret the DDP rollout as a \emph{local sensitivity-based correction to the Newton search direction ($\bm{\delta u}^L$)}. Generalizing this interpretation to the constrained setting, consider the following control law:
\begin{small}
\begin{equation*}
    \delta u_k[\alpha] =\mathrm{clip}_{\ulb - u_k}^{\uub-u_k}\bigg[\alpha \delta \pi_k^*(\delta x_k^L) + \dfrac{\partial \delta \pi_k^*(\alpha \delta x_k^L)}{\partial \delta x_k} (\delta x_k[\alpha] - \alpha \delta x_k^L)\bigg],
\end{equation*}
\end{small}
where similarly to~\eqref{lin_rollout}, $\bm{\delta x}^L$ is obtained from rolling-out $\{\delta \pi_k^*(\delta x_k^L)\}_{k=0}^{N-1}$ through the linearized dynamics. 
Now, it follows from Bellman's principle of optimality that $\delta \pi_k^*(\delta x_k^*) = \delta u_k^*$, where $(\bm{\delta x}, \bm{\delta u})^*$ are the optimal solution to~\eqref{qp_prob_implicit}. Further since $\delta x_0^L = \delta x_0^* = 0$, it follows inductively that $\bm{\delta x}^L = \bm{\delta x}^*$. Thus, the final control law for $\delta u_k[\alpha]$ becomes
\begin{equation}
    \mathrm{clip}_{\ulb - u_k}^{\uub-u_k}\bigg[\alpha \delta u_k^*+
    \dfrac{\partial \delta \pi^*_k(\alpha\delta x_k^*)}{\partial \delta x_k} (\delta x_k[\alpha] - \alpha\delta x_k^*)\bigg].
\label{ddp_ideal}
\end{equation}
Given that $\delta \pi^*_k$, as defined in~\eqref{pi_pert}, is implicitly the solution of a \emph{variable-horizon} optimization, it is computationally prohibitive to compute the sensitivity matrices above via explicit differentiation. Instead, we next define a DP recursion to exactly compute these sensitivities about the \emph{fixed} sequence $\bm{\delta x}^*$.

\begin{remark}
Since problem~\eqref{pi_pert} is a multi-parametric QP in $\delta x_k$, $\delta \pi^*_k$ is a piecewise-affine function of $\delta x_k$~\cite{BorrelliBemporad2017}. Thus, for $\alpha=1$, the expression inside the brackets in~\eqref{ddp_ideal} lies within $\delta \U_k$ only for $\delta x_k[1]$ in a local region around $\delta x_k^*$, thereby necessitating the clipping operation (i.e., projection onto $\delta \U_k$).
\end{remark}

\subsection{DP Recursion for Computing Sensitivity Gains}
\label{sec:sens-exact}

We outline the DP recursion first and characterize its correctness in Theorem~\ref{thm:sens-gains}.
\medskip

\noindent {\bf Initialization}: Set $P_N = Z_N, p_N = q_N, v_N = 0, G^{\mathrm{cr}}_N = -J_N^x, h^{\mathrm{cr}}_N = c_N^x$.
\medskip

\noindent {\bf Time-step $k = N-1, \ldots, 0$}: Define the function $\delta \tilde{V}_{k+1}(\delta x_{k+1}):= (1/2)\delta x_{k+1}^T P_{k+1} \delta x_{k+1} + p_{k+1}^T \delta x_{k+1} + v_{k+1}$. For $\delta x_k \in \delta \X_k$, consider the one-step QP:
\begin{equation} \small
    \begin{split}
        \min_{\substack{\delta u_{k} \\ \delta x_{k+1}}} \quad &\tilde{l}_{k}(\delta x_k, \delta u_k) + \delta \tilde{V}_{k+1}(\delta x_{k+1}) \\
        \mathrm{s.t.} \quad &\delta x_{k+1} = A_k \delta x_{k} + B_{k} \delta u_{k} \\
        & \underbrace{\begin{bmatrix} -J_k^u & 0 \\ 0 & G^{\mathrm{cr}}_{k+1} \end{bmatrix}}_{:=\bar{G}_k} \begin{bmatrix} \delta u_{k} \\ \delta x_{k+1} \end{bmatrix} \leq \begin{bmatrix} c^u_{k} \\ h^{cr}_{k+1} \end{bmatrix} := \bar{h}_{k}.
    \end{split}
\label{prob_k}
\end{equation}
Let $\delta \hat{\pi}^*_{k}(\delta x_k)$ and $\hat{y}^*_k(\delta x_k)$ denote the optimal control perturbation and inequality dual solutions for the above one-step QP, as a function of $\delta x_k$. Define the sensitivity matrices $K_k^u$ and $K_k^y$ as the Jacobians ~\cite{AgrawalAmos2019} of $\delta \hat{\pi}^*_{k}$ and $\hat{y}^*_k$ respectively, evaluated at $\delta x_k^*$,  and define the affine functions:
\begin{subequations}
\begin{align}
\delta \hat{\pi}^*_{k,a}(\delta x_{k}) &= \delta \hat{\pi}^*_{k}(\delta x_{k}^*) + K_{k}^u (\delta x_{k} - \delta x_{k}^*) \\
\hat{y}^*_{k,a}(\delta x_{k}) &= \hat{y}^*_{k}(\delta x_{k}^*) + K_{k}^y (\delta x_{k} - \delta x_{k}^*).
\end{align}
\label{sens_aff}
\end{subequations}

\noindent{\bf Recurse}: Compute:
\begin{equation}\small
    G^{\mathrm{cr}}_{k}:=  \begin{bmatrix}  -J_{k}^{x} \\ \bar{G}_{k} \begin{bmatrix} K^u_{k} \\ \bar{A}_k \end{bmatrix} \\ -K_{k}^y \end{bmatrix}, \  h^{\mathrm{cr}}_{k} = \begin{bmatrix} c_k^x \\ \bar{h}_{k} - \bar{G}_{k} \begin{bmatrix} I \\ B_{k} \end{bmatrix} \delta\hat{\pi}^*_{k,a}(0) \\
    \hat{y}^*_{k,a}(0) \end{bmatrix},
\label{ctg_cons}
\end{equation}
where $\bar{A}_k := A_k + B_k K_k^u$, and:
\begin{subequations}\small
\begin{align}
    \bar{r}_k &= r_k + B_k^T p_{k+1} \quad \bar{R}_k = Z_{k, uu} + B_{k}^T P_{k+1} B_{k} \\
    \bar{M}_{k} &= Z_{k, xu} + A_{k}^T P_{k+1} B_{k} \\
    p_{k} &= q_k + A_{k}^T p_{k+1} + K_{k}^{u^T} \bar{r}_{k} + (K_{k}^{u^T} \bar{R}_{k}+ \bar{M}_{k}) \delta\hat{\pi}^*_{k,a}(0) \\
    P_{k} &= Z_{k, xx} + A_{k}^T P_{k+1} A_{k} + K_{k}^{u^T} \bar{R}_{k} K_{k}^{u^T} + \nonumber \\
    &\quad + \bar{M}_{k} K_{k}^{u} + K_{k}^{u^T} \bar{M}_{k}^T \\
    v_k &= \bar{r}_{k}^T \delta\hat{\pi}^*_{k,a}(0) + \dfrac{1}{2} \delta\hat{\pi}^*_{k,a}(0)^T \bar{R}_{k} \delta\hat{\pi}^*_{k,a}(0) + v_{k+1}.
\end{align}
\label{ctg_recurse}
\end{subequations}
We now characterize the correctness of this DP recursion in the following theorem; the proof is provided in Appendix~\ref{app:proof}.

\begin{theorem}\label{thm:sens-gains}
Suppose that for each $k \in \{N-1,\ldots, 0\}$, the solution of the one-step QP in~\eqref{prob_k} at $\delta x_k = \delta x_k^*$ satisfies Linear Independent Constraint Qualification (LICQ). Then, the recursion in~\eqref{ctg_cons}--\eqref{ctg_recurse} is well-defined. Define the sets $\CR_k: \{\delta x_k: G^{\mathrm{cr}}_k \delta x_k \leq h^{\mathrm{cr}}_k\}$, for $k=N,\ldots, 0$. It holds that:
\[
    \begin{split}
    &\delta x_k^* \in \CR_k \ k=N,\ldots, 0 \ \ \text{and} \\
    &\begin{cases} \forall \delta x_k \in \CR_k &: \  \delta \pi^*_k(\delta x_k) = \delta \hat{\pi}^*_{k,a}(\delta x_k) \\
                   \forall \delta x_{k+1} \in \CR_{k+1} &: \  \delta V_{k+1}(\delta x_{k+1}) = \delta \tilde{V}_{k+1}(\delta x_{k+1}),
     \end{cases}
     \end{split}
\]
for $k=N-1,\ldots, 0$.
\end{theorem}

\begin{remark}
A notable consequence of Theorem~\ref{thm:sens-gains} is that the canonical cost-to-go recursion is ill-posed in the presence of constraints. One must back-propagate both the cost-to-go terms and a set of constraints (i.e., the sets $\{\CR_k\}$) that define the regions where the quadratic models of the cost-to-go functions are precise.
\end{remark}

Despite the exactness of the DP recursion, there are some computational drawbacks. First, one must solve both the ``full-horizon" QP defined in~\eqref{qp_prob_implicit} and the one-step QPs defined in~\eqref{prob_k}, serially. Second, back-propagating sets $\{\CR_k\}$ is not numerically robust, particularly if the sensitivity $K_k^y$ is ill-defined. This occurs when the LICQ condition fails and the resulting matrix solve computation for the sensitivities is singular. Thus, in the next section, we outline a parallelized and tuneable approximation to the sensitivity gains, derived from the viewpoint of interior point methods.

\subsection{Approximating the Sensitivity Gains}
\label{sec:sens-approx}

For $k=N-1,\ldots, 0$, define problem $\Pp_{k\multimap}(\delta x)$ as the tail portion of QP sub-problem~\eqref{qp_prob_implicit}, starting at time-step $k$ at $\delta x$. Let $\bm{\delta u}^*_{k\multimap}(\delta x)$ represent the optimal solution as a function of $\delta x$, i.e., the optimal control perturbation sequence starting at time-step $k$. Notice that $\delta \pi_k^*(\delta x)$, as defined in~\eqref{pi_pert}, corresponds to the first element of the sequence $\bm{\delta u}^*_{k\multimap}(\delta x)$.

Now define the QP problem $\Pp_k(\delta x)$ as QP sub-problem~\eqref{qp_prob_implicit}, subject to an \emph{additional equality constraint}: $\delta x_k = \delta x$, and let $\bm{\delta u}^*_{k:}(\delta x)$ represent the optimal tail control perturbation sequence starting at time-step $k$.

Notice then that for all $\delta x$ where $\Pp_k(\delta x)$ is feasible, we have that $\bm{\delta u}^*_{k\multimap}(\delta x) = \bm{\delta u}^*_{k:}(\delta x)$. Thus, $\delta \pi_k^*(\delta x)$ is equal to the optimal control perturbation at time-step $k$ for problem $\Pp_k(\delta x)$, hereby denoted as the function $\delta \hat{\pi}_k^*(\delta x)$. Further, since $\Pp_k(\delta x_k^*)$ is feasible, one may compute the desired sensitivity gains $K_k^u$ as the Jacobian $\partial \delta \hat{\pi}_k^*(\delta x_k^*)/\partial \delta x$.

Note the distinction: the sensitivity $\partial \delta \hat{\pi}_k^*/\partial \delta x$ corresponds to the Jacobian of the solution of a \emph{fixed-horizon} QP (problem $\Pp_k(\delta x)$) w.r.t. a parameter ($\delta x$) that defines the equality constraint $\delta x_k = \delta x$. In comparison, the sensitivity $\partial \delta \pi_k^*/\partial \delta x$ corresponds to the Jacobian of the solution of a \emph{variable-horizon} QP (problem $\Pp_{k\multimap}(\delta x)$) w.r.t. a parameter ($\delta x$) that defines the ``initial condition." The former computation is easily parallelized.

Leveraging a recent result in~\cite{JinMou2021}, we approximate $\partial \delta \hat{\pi}_k^*(\delta x_k^*)/\partial \delta x$ by the Jacobian of the solution of the following unconstrained barrier re-formulation of problem $\Pp_k(\delta x)$ w.r.t. $\delta x$ at $\delta x = \delta x_k^*$:
\begin{equation}
\begin{split}\small
    \min_{\bm{\delta u}, \bm{\delta x}} \  &\sum_{j=0}^{N-1} \bigg[\tilde{l}_j - \gamma \bm{1}^T \log(c_j^x + J_j^x \delta x_j) - \gamma\bm{1}^T\log(c_j^u + J_j^u \delta u_j)\bigg] \nonumber \\
    &\  + \dfrac{1}{2 \gamma} \|\delta x_k - \delta x\|^2 + \tilde{l}_N - \gamma \bm{1}^T \log(c_N^x + J_N^x \delta x_N),
\end{split}
\label{barrier_prob}
\end{equation}
subject to the linear dynamics in~\eqref{lin_dyn}; where $\gamma > 0$ is the barrier constant. Denote $K_k^u(\gamma)$ to be the barrier-based Jacobian with parameter $\gamma$ and let $K_k^u$ be the true Jacobian. Under appropriate conditions on the solution of QP sub-problem~\eqref{qp_prob_implicit}, $K_k^u(\gamma) \rightarrow K_k^u$ as $\gamma \rightarrow 0$~\cite{JinMou2021}. 

Practically, we compute the Jacobians $K_k^u(\gamma)$ efficiently using iLQR and a straightforward application of the Implicit Function Theorem~\cite{AmosJimenez2018}. We initialized the solver with the QP sub-problem~\eqref{qp_prob_implicit} solution $\bm{\delta u}^*$, and found only a handful of iterations were needed to converge, particularly since problem~\eqref{barrier_prob} is convex.

\subsection{Hybrid SQP Algorithm}

We formally state the Hybrid SQP algorithm as a line-search modification of Shooting SQP, introduced in Section~\ref{sec:shooting}. Thus, at the current primal-dual iterate $(\bm{u}, \bm{x}[\bm{u}],\bm{y})$:
\medskip

\noindent \emph{Step 1}: Solve QP-subproblem~\eqref{qp_prob_implicit} to obtain the optimal perturbation sequence pair $(\bm{\delta u}^*, \bm{\delta x}^*)$.
\medskip

\noindent \emph{Step 2}: Compute the sensitivity gains $\{K_k\}_{k=0}^{N-1}$, using either the DP recursion in Section~\ref{sec:sens-exact} (i.e., $K_k = K_k^u$), or the smoothed approximation method in Section~\ref{sec:sens-approx} (i.e., $K_k = K_k^u(\gamma)$ for some $\gamma >0$).
\medskip

\noindent \emph{Step 3}: Perform line-search using~\eqref{line-search}, where $\bm{x}[\alpha] := \bm{x} + \bm{\delta x}[\alpha]$ is now defined by the \emph{closed-loop} rollout:
\[
\begin{split}
    \delta u_k[\alpha] :=&\mathrm{clip}_{\ulb - u_k}^{\uub - u_k}\left[\alpha \delta u_k^* + K_k (\delta x_k[\alpha] - \alpha \delta x_k^*)\right] \\
    \delta x_{k+1}[\alpha] &= f(x_k + \delta x_k[\alpha], u_k + \delta u_k[\alpha]) - x_{k+1}, \ \delta x_0[\alpha] = 0.
\end{split}
\]
Notice that if $\alpha=1$, the rollout corresponds with the ideal DDP rollout in~\eqref{ddp_ideal}, while for $\alpha < 1$, we end up with an approximation\footnote{As the sensitivity gains are only valid in a neighborhood of $\bm{\delta x}^*$, it is possible (though rare in our experiments) that the computed step-length $\alpha$ falls below the user-set threshold $\underline{\alpha}$ for a specific iteration. As a backup, we compute a set of TV-LQR gains $\{K_k^{\mathrm{lqr}}\}$ using the linearized dynamics and the Hessian of the objective function $\mathcal{J}(\bm{u}, \bm{x})$, and perform the closed-loop rollout with these gains. This strategy is motivated by the goal of tracking the perturbation $\alpha \bm{\delta x}^*$ during the rollout~\cite{TennyWright2004}.} stemming from using a fixed gain matrix.

\section{Experiments} \label{sec:exp}

We compare the na{\"i}ve, open-loop, Shooting SQP implementation introduced in Section~\ref{sec:shooting} (referred to as $\ol$) with the DDP-style closed-loop variation developed in Section~\ref{sec:dp-sqp} (referred to as $\cl$ and $\clg$) on two environments. The identifiers $\cl$ and $\clg$ distinguish between the exact DP recursion and the smoothed barrier-based approximation for computing the forward rollout gains. Please see Appendix~\ref{app:numerics} for details regarding problem setup, SQP hyperparameters, additional plots, and an extra worked example.

\subsection{Motion Planning for a Car}
The first example is taken from~\cite{AoyamaBoutselis2020}, featuring a 2D car ($n=4$, $m=2$) moving within the obstacle-ridden environment shown in Figure~\ref{fig:car-solns}. The objective is to drive to the goal position $(3, 3)$ with final velocity $0$ and orientation aligned with the horizontal axis in $N=40$ steps, while avoiding the obstacles. 

Figure~\ref{fig:car-solns} shows the computed trajectories by the three methods for three different initial conditions, while Table~\ref{tab:car-stats} provides the solver statistics. Notice that the $\ol$ method fails to converge within 100 iterations (the limit) for two of the three cases. In contrast, both closed-loop variations are quickly able to converge to stationary solutions.
\begin{table}[h]
\centering
\resizebox{0.5\textwidth}{!}{\begin{tabular}{l | c | c c c c c}
\toprule
Method & Case & Converged & Iter & Obj & Viol & Time/it [s] \\
\midrule
\multirow{3}{*}{$\ol$} 
& 1 & \ding{55} & 100 & 278.53 & -0.0346 & 0.34 \\
& 2 & \ding{55} & 100 & 2.17 & 0.0061 & 0.30 \\
& 3 & \checkmark & 12 & 21.49 & 3.25e-6 & 0.30 \\
\midrule
\multirow{3}{*}{$\cl$}
& 1 & \checkmark & 19 & 3.19 & 7.87e-6 & 13.17 \\
& 2 & \checkmark & 19 & 6.98 & 1.43e-6 & 19.19 \\
& 3 & \checkmark & 13 & 21.58 & 1.12e-5 & 17.90 \\
\midrule
\multirow{3}{*}{$\clg$}
& 1 & \checkmark & 19 & 3.19 & 1.34e-5 & 0.40 \\
& 2 & \checkmark & 16 & 2.06 & 1.28e-5 & 0.41 \\
& 3 & \checkmark & 11 & 21.58 & -4.67e-6 & 0.38 \\
\bottomrule
\end{tabular}}
\caption{\footnotesize{Solver statistics for the car planning example. \emph{Baseline:} $\ol$, \emph{Ours}: $\{\cl, \clg\}$. \emph{Iter}, \emph{Obj}, and \emph{Viol} report the number of iterations, objective, and max state constraint violation (negative value indicates infeasibility), respectively, for the final solution. \emph{Time/it} reports the average (over the course of the optimization) computation time per SQP iteration.}}
\label{tab:car-stats}\vspace{-0.5cm}
\end{table}

\begin{figure}[h]
  \centering
  \includegraphics[width=0.35\textwidth]{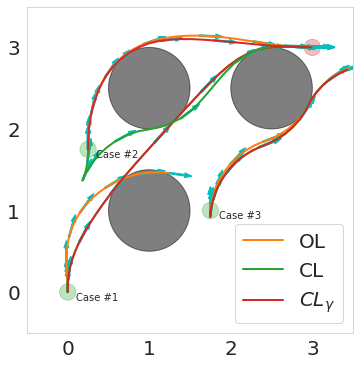}  
\caption{\footnotesize{Solution trajectories for all solvers on the car planning example. Start/end locations shaded green/red.}}
\label{fig:car-solns}\vspace{-0.5cm}
\end{figure}

\begin{figure}[h]
  \centering
  \includegraphics[width=0.5\textwidth]{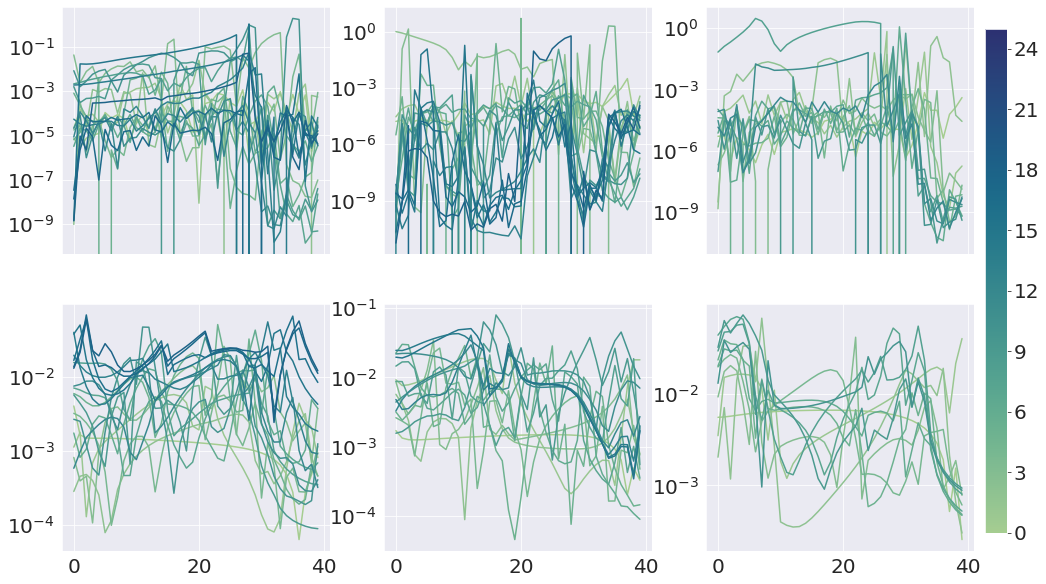}  
\caption{\footnotesize{Re-construction error between open-loop solution $\bm{\delta u}^*$ and perturbation policy solution $\delta \hat{\pi}_k^*(\delta x_k^*)$ as a function of time-step across all SQP iterations. \emph{Top}: $\cl$; \emph{Bottom}: $\clg$, $\gamma=10^{-4}$; \emph{Left-to-right}: Case index. Colorbar identifies the SQP iteration.}}
\label{fig:car-errs}\vspace{-0.7cm}
\end{figure}
In Figure~\ref{fig:car-errs}, we plot the re-construction error $\|\delta \hat{\pi}_k^*(\delta x_k^*) - \delta u_k^*\|$ for $\cl$ and $\clg$ over the course of the SQP iterations. Observe that for most iterations, the error is negligible for $\cl$, with occasional spikes resulting from the numerical instability of the \emph{constrained} DP recursion. In contrast, the error remains sufficiently low for \emph{all} iterations of $\clg$, even leading to a better quality (lower objective) solution for Case \#2. We hypothesize that the better numerical stability of $\clg$ stems from a \emph{smoothing} of the computed Jacobians (i.e., feedback gains), courtesy of the unconstrained barrier re-formulation. Finally, a notable advantage of $\clg$ over $\cl$ is the computation time. While $\cl$ involves differentiating through the KKT conditions of one-step horizon QPs, this computation must happen serially in the backward pass. In contrast, $\clg$ computes the required Jacobians across all time-steps in parallel using an efficient adjoint recursion associated with problem~\eqref{barrier_prob}. Consequently, the computation times-\emph{per iteration} are much closer together for $\ol$ and $\clg$ than for $\ol$ and $\cl$. For the remaining experiments, we only compare $\ol$ and $\clg$.

\subsection{Quad-Pendulum}
Consider a quadrotor with an attached pendulum ($n=8$, $m=2$) moving within an obstacle-cluttered 2D vertical plane, subject to viscous friction at the pendulum joint. The system is subject to operational constraints on the state and control input, as well as obstacle avoidance constraints. The task involves planning a trajectory starting at rest with the pendulum at the stable equilibrium, to a goal location, with the pendulum upright and both quadrotor and pendulum stationary. Table~\ref{tab:quadpole} provides the solver statistics for $\clg$ and $\ol$ (up until the algorithm stalls due to infeasibility of the QP sub-problem). Figure~\ref{fig:quadpole2} shows a timelapse of the solution for the more difficult of the two cases, highlighting the ability of $\clg$ in solving challenging planning tasks.

\begin{table}[H]
\centering
\resizebox{0.5\textwidth}{!}{\begin{tabular}{l | c | c c c c c}
\toprule
Method & Case & Converged & Iter & Obj & Viol & Time/it [s] \\
\midrule
\multirow{2}{*}{$\ol$} 
& 1 & \ding{55} & \texttt{stall} (2) & 2272.8 & -7.33  &  1.49\\
& 2 & \ding{55} & \texttt{stall} (4) & 488.9 & -0.58 & 1.35\\
\midrule
\multirow{3}{*}{$\clg$}
& 1 & \checkmark & 39 & 9.31 & -2.73e-10 & 6.25 \\
& 2 & \checkmark & 59 & 11.57 & -4.33e-10 & 5.65 \\
\bottomrule
\end{tabular}}
\caption{\footnotesize{Solver statistics for the Quad-Pendulum example. Algorithm $\ol$ stalls within a handful of iterations.}}
\label{tab:quadpole}\vspace{-0.9cm}
\end{table}.

\section{Conclusions}
In this work, we re-interpret DDP rollout policies from a perspective of sensitivity-based corrections, and use this insight to develop algorithms for computing analogous policies for constrained problems. We incorporate the resulting closed-loop rollouts within a shooting-based SQP framework, and demonstrate significant improvements in convergence speed over a standard SQP implementation using open-loop rollouts.

Our work opens several avenues for future research. First, a key bottleneck of SQP involves solving the QP sub-problem at each iteration to compute a ``nominal" perturbation sequence. This may be achieved with fast, but potentially, less-accurate unconstrained solvers (e.g., augmented-Lagrangian iLQR), that additionally compute the desired sensitivity gains using an efficient Riccati recursion. Second, leveraging recent results on differentiating through the solution of \emph{general} convex problems, the sensitivity-based computations may be applied to the sequential-\emph{convex}-programming algorithm. Finally, while the SQP algorithm was studied in the shooting context, recent work~\cite{MastalliEtAl2020} has demonstrated how to incorporate nonlinear rollouts with both states and controls as optimization variables, albeit in an otherwise unconstrained setting. The sensitivity-based gain computation can be extended to this setting, potentially boosting the efficiency of ``full" direct methods.

\bibliographystyle{IEEEtran}
\bibliography{main_ieee.bbl}


\newpage

\onecolumn

\appendices

\section{Related Work}
\label{app:related}
iLQR/DDP-like methods for constrained trajectory optimization fall into one of three main categories: control-bounds only~\cite{Tassa2014},\cite{MartiSolaEtAl2020},  modification of the backward pass via KKT analysis~\cite{GiftthalerBuchli2017,AoyamaBoutselis2020,XieLiu2017,Yakowitz1986,LantoineRussell2012}, and augmented Lagrangian methods~\cite{PlancherManchester2017,HowellJackson2019,SindhwaniRoelofs2017,MaCheng2020,PavlovShames2021}. 
In the first category, only control-bound constraints are considered. For instance,~\cite{Tassa2014} leverage the box-constrained DDP algorithm where the projected Newton algorithm is used to compute the affine perturbation policy in the backward pass, accounting for the box constraints on the control input. On the other hand, \cite{MartiSolaEtAl2020} composes a ``squashing" function to constrain the input, and augments the objective with a barrier penalty to encourage the iterates to stay away from the plateaued regions of squashing function.

More general constraints are considered in the second and third categories~\cite{GiftthalerBuchli2017,AoyamaBoutselis2020,XieLiu2017,Yakowitz1986,LantoineRussell2012, PlancherManchester2017,HowellJackson2019,SindhwaniRoelofs2017,MaCheng2020,PavlovShames2021}. The work in~\cite{GiftthalerBuchli2017} features equality constraints. The update step leverages constraint linearization and a nullspace projection to reduce the problem to a singular optimal control problem in an unconstrained control input lying within the nullspace of the linearized constraints. The works in~\cite{AoyamaBoutselis2020,XieLiu2017,Yakowitz1986,LantoineRussell2012} formulate a \emph{constrained} backward pass, where the one-step optimization problems now feature general state and control inequality constraints, linearized about the current trajectory iterate, yielding constrained one-step QPs. The algorithm in~\cite{XieLiu2017} extracts a guess of the active inequality constraints at each time-step by looking at the current trajectory iterate, and formulates the backward pass using linearized active constraints as equalities. The resulting KKT system is solved using Schur's complement and using the dual variable extracted from the ``feedforward" perturbation (i.e., the control perturbation computed assuming zero state perturbation) to yield an affine perturbation policy. The work in~\cite{Yakowitz1986} formulates the theoretical underpinnings of such a constrained backward pass by introducing the stage-wise KKT conditions, and implements a very similar algorithm to~\cite{XieLiu2017}, but where the active set is guessed to be equal to the active set at the feedforward solution, computed using linearized inequality constraints.~\cite{LantoineRussell2012} implements an identical backward pass, but with the active set guessed by extracting the violating constraints from implementing the feedforward perturbation (computed using a trust-region constraint). More recently,~\cite{AoyamaBoutselis2020} leverages a slack-form of the linearized inequality constraints in the one-step optimization problems and uses an iterative procedure to refine the dual variable assuming zero state perturbation. A final solve of the KKT system using Schur's complement yields the affine perturbation policy. The forward pass of all these methods discard this affine policy and instead use the one-step QPs to compute the control perturbation trajectory. In similar spirit,~\cite{PavlovShames2021} introduced an interior-point variation of the backward pass, where the one-step optimization is re-written in min-max form w.r.t. the Lagrangian, and linearization of the associated perturbed KKT conditions is used to compute the affine perturbation policy.

The algorithms in~\cite{PlancherManchester2017,HowellJackson2019,SindhwaniRoelofs2017,MaCheng2020} incorporate the constraints by forming the augmented Lagrangian, and alternate between unconstrained trajectory optimization (with the augmented Lagrangian as the objective), and updating the dual variables and penalty parameter -- loosely mimicking the method of multipliers~\cite{Bertsekas2016}.~\cite{PlancherManchester2017} implements this in combination with a sampling-based construction of the optimal cost-to-go approximation;~\cite{HowellJackson2019} dovetails the multipliers algorithm with Newton-based projections to project the intermediate solution onto the current active set; ~\cite{SindhwaniRoelofs2017} adopts an ADMM-based solution, by leveraging indicator function representations of the constraints and alternates between an unconstrained TV-LQR problem, linearized constraint projection, consensus update, and dual update.

In contrast to the methods above, our approach does not involve guessing active constraint sets for the one-step QPs and instead re-interprets the DDP-style control law from the lens of sensitivity analysis about a known feasible perturbation sequence. We use this interpretation to derive an augmented ``backward pass" where we additionally back-propagate a ``next-step" constraint set, along with the quadratic cost-to-go parameters. Further, we demonstrate how to approximate the desired feedback gains across all time-steps in parallel, dispelling the need for backward passes which are numerically less robust. Finally, we embed the closed-loop rollout within a theoretically sound SQP framework.

\section{SQP Preliminaries}
\label{app:sqp}

In this appendix, we outline the main steps of the SQP algorithm. These steps represent a simplified version of the commercial code NPSOL~\cite{GillMurray1986}, and use the termination conditions from the commercial code SNOPT~\cite{GillMurray2005}. Consider the following smooth optimization problem:
\begin{equation}
\begin{split}
    \Prob: \quad \min_{x \in \reals^n} \quad & f(x) \\
    \mathrm{s.t.} \quad & c(x) \in \reals^m_{\geq 0},
\end{split}
\label{generic-opt}
\end{equation}
where the functions $f$ and $c$ are at least $\mathcal{C}^2$. Define the Lagrangian as:
\[
    \mathcal{L}(x, y) = f(x) - y^T c(x).
\]
A standard SQP iteration consists of the following steps:
\begin{itemize}
    \item Solve a quadratic problem (QP) at current primal-dual point $(x, y)$:
    \[
    \begin{split}
        \delta x^* = \argmin_{\delta x} \quad & g(x)^T \delta x + \dfrac{1}{2} \delta x^T H(x, y) \delta x \\
        \mathrm{s.t.} \quad & c(x) + J(x) \delta x \geq 0.
    \end{split}
    \]
    where $g(x) = \nabla_x f(x)$, $H(x, y) = \nabla^2_{xx} \mathcal{L}(x, y)$, and $J(x) = \partial c(x)/\partial x$. Let $\hat{y}$ be the optimal dual variable for the inequality constraint above, and define the primal-dual search directions $(\delta x^*, \delta y^*)$, where $\delta y^* := \hat{y} - y$.
    
    \item Define the Merit function to be the augmented Lagrangian:
    \[
    \mathcal{M}(x, y, s; \rho) = f(x) - y^T (c(x) - s) + \dfrac{1}{2}\rho \|c(x) - s\|^2,
    \]
    where $s \in \reals^m_{\geq 0}$ is the inequality ``slack," defined only for the line-search, and $\rho$ is a non-negative penalty parameter. At the current iterate $(x, y)$ and penalty value $\rho$, set $s$ as follows:
    \[
        s = \begin{cases} \max(0, c(x)) \quad &\text{if} \ \rho = 0 \\
        \max\left( 0, c(x) - \dfrac{1}{\rho}y \right)\quad &\text{otherwise}
        \end{cases},
    \]
    where the $\max$ operator is component-wise, and set the slack search direction $\delta s^*$ as $c(x) + J(x) \delta x^* - s$.
    
    \item Define the line-search function $\phi(\alpha; \rho) := \mathcal{M}(x + \alpha \delta x^*, y + \alpha \delta y^*, s + \alpha \delta s^*; \rho)$ and pick the updated penalty parameter $\rho^+$ s.t. $\phi'(0; \rho^+) \leq -(1/2)\Delta^*$, where $\Delta^*$ is the decrement, defined as: $\delta x^{*^T} H(x, y) \delta x^*$. A simple update rule is given below~\cite{MurrayPrieto1995}: 
    \begin{equation}
    \rho^+ \leftarrow \begin{cases} \rho \quad &\text{if} \quad \phi'(0; \rho) \leq -\dfrac{1}{2} \Delta^* \\
    \max\left\{2\rho, \dfrac{g(x)^T \delta x + \frac{1}{2}\Delta^* + (2y - \hat{y})^T (c(x) - s)}{\|c(x) - s \|^2} \right\} \quad &\text{otherwise}.
    \end{cases}
    \label{rho_up}
    \end{equation}
    
    \item Compute (e.g., using backtracking, safe-guarded polynomial interpolation, etc.) the largest step length $\alpha \in [\underline{\alpha}, 1]$, where $\underline{\alpha}$ is a user-specified lower-bound, s.t. the following line-search conditions are satisfied:
    \begin{equation}
        \phi(\alpha; \rho^+) - \phi(0; \rho^+) \leq \sigma \alpha \phi'(0; \rho^+), \quad \text{and} \quad |\phi'(\alpha; \rho^+)| \leq -\eta \phi'(0; \rho^+),
        \label{ls_condn}
    \end{equation}
    where $0 < \sigma \leq \eta < \frac{1}{2}$. 
    
    \item Update $(x, y)$ using the computed step length, and set $\rho \leftarrow \rho^+$.
\end{itemize}
Note that this stripped-down version of SQP does not account for infeasibility detection, which involves slackened forms of the QP sub-problem.

The termination conditions are based upon specified relative tolerances $\tau_p, \tau_d \in \reals_{>0}$. Define $\tau_x := \tau_p (1 + \|x\|)$, and $\tau_y := \tau_d (1 + \|y\|)$. Then, convergence to a KKT stationary point is declared and the algorithm terminated if:
\begin{equation}
    \begin{cases}
    \min_i c_i(x) \geq -\tau_x \\
    \min_j y_j \geq -\tau_y \\
    \|c(x) \circ y\|_{\infty} \leq \tau_y \\
    \|g(x) - J(x)^T y \|_{\infty} \leq \tau_y
    \end{cases}
\label{sqp-term}
\end{equation}
where $\circ$ denotes the Hadamard product.

\section{Details for Shooting SQP}
\label{app:ocp}

We provide here the explicit expressions used within various stages of Shooting SQP. 

\subsection{QP Sub-Problem}

The linear term in~\eqref{qp_obj_implicit} and~\eqref{qp_obj_explicit} is given as follows:

\begin{equation}
\begin{split}
\langle\bm{\delta u}, \nabla_{\bm{u}} \J_R(\bm{u})\rangle &= \sum_{k=0}^{N-1} \nabla_x l_k^T \delta x_k + \nabla_u l_k^T \delta u_k + \nabla_x l_N ^T \delta x_N \\
&:= \sum_{k=0}^{N-1} q_k^T \delta x_k + r_k^T \delta u_k + q_N^T \delta x_N
\end{split}
\label{qp_obj_explicit_linear}
\end{equation}
where $\delta x_k$ is defined recursively using the linearized dynamics:
\begin{equation}
        \delta x_0 = 0, \ \ \delta x_{k+1} = A_k \delta x_k + B_k \delta u_k, \quad k=0,\ldots, N-1.
\end{equation}
The Hessian term in~\eqref{qp_obj_implicit} and~\eqref{qp_obj_explicit} is given by:
\begin{equation}
    \begin{split}
        \langle \bm{\delta u}, \nabla_{\bm{u}}^2 \mathcal{L}_R(\bm{u}, \bm{y}) \bm{\delta u} \rangle = &\sum_{k=0}^{N-1} \begin{bmatrix} \delta x_k \\ \delta u_k \end{bmatrix}^T \overbrace{\begin{bmatrix} \nabla_{xx}^2 \hat{H}_k & \nabla_{xu}^2 \hat{H}_k \\ * & \nabla_{uu}^2 \hat{H}_k \end{bmatrix}}^{:=Z_k} \begin{bmatrix} \delta x_k \\ \delta u_k \end{bmatrix} + \delta x_N ^T \overbrace{\nabla_{xx}^2 \hat{l}_N}^{:=Z_N} \delta x_N, \\
        &\qquad \qquad \text{where:} \\
        &\hat{l}_k(x_k, u_k, y_k) := l_k(x_k, u_k) - y_k^T c_k(x_k, u_k), \quad k=0,\ldots, N-1 \\ &\hat{l}_N(x_N, y_N) := l_N(x_N) - y_N^T c_N(x_N) \\
        &\hat{H}_k = \hat{l}_k(x_k, u_k, y_k) + \hat{\nu}_{k+1}^T f(x_k, u_k), \quad k = 0,\ldots, N-1\\
        &\begin{cases} \hat{\nu}_N = \nabla_x \hat{l}_N(x_N, y_N) \\
        \hat{\nu}_k = \nabla_x \hat{l}_k(x_k, u_k, y_k) + A_k^T \hat{\nu}_{k+1}, \  k=N-1,\ldots, 1
        \end{cases}
    \end{split}
    \label{qp_obj_implicit_hess}
\end{equation}

\begin{remark}
The matrices $Z_k$ above may not be positive semi-definite. To ensure problem~\eqref{qp_prob_implicit} is convex, we project $Z_k$ to the positive semi-definite cone with some $\epsilon > 0$ tolerance.
\end{remark}

\subsection{Line-Search}
\label{app:rho}
Define $s_k = (s_k^x, s_k^u)$ where $s_k^x$ and $s_k^u$ have the same dimensionality as $c_k^x$ and $c_k^u$ respectively. We use the augmented Lagrangian function as the merit function for line-search:
\begin{equation}
        \mathcal{M}_I(\bm{u}, \bm{y}, \bm{s}; \rho) = \sum_{k=0}^{N-1} l_k - y_k^T(c_k - s_k) + \dfrac{\rho_k}{2}\|c_k - s_k\|^2  + l_N - y_N^T(c_N - s_N) + \dfrac{\rho_N}{2}\|c_N - s_N\|^2,
\label{merit_fnc}
\end{equation}
where the dependence on $\bm{x}$ is implicit, i.e., $\bm{x} = \bm{x}[\bm{u}]$. The non-negative constants $\{\rho_k\}_{k=0}^{N}$ are a set of penalty parameters, which for brevity, we denote by $\rho$. The vectors $s_k$ represent a set of slack vectors, used only for the line-search. At the current iterate $(\bm{u}, \bm{x})$ and set of penalty parameters $\rho$, set $s_k$ as
\[
    s_k = \begin{cases} \max(0, c_k) \quad &\text{if} \ \rho_k = 0 \\
        \max(0, c_k - y_k/\rho) \quad &\text{otherwise}
        \end{cases},
\]
and set $\delta s_k^* = (\delta s_k^{x^*}, \delta s_k^{u^*})$, the slack-search directions, as:
\[
\begin{split}
    \delta s_k^{x^*} &= c_k^x(x_k) + J_k^x \delta x_k^* - s_k^x \\
    \delta s_k^{u^*} &= c_k^u(u_k) + J_k^u \delta u_k^* - s_k^u.
\end{split}
\]
As in Appendix~\ref{app:sqp}, we now must compute an updated set of penalty parameters $\rho^+$ so that $\phi'(0; \rho^+) \leq -(1/2) \langle \bm{\delta u}^*, \nabla_{\bm{u}}^2 \mathcal{L}_R(\bm{u}, \bm{y}) \bm{\delta u}^* \rangle := -(1/2) \Delta^*$. To derive this update, let us consider an abstract optimization problem of the form:
\[
\begin{split}
    \min_{x \in \reals^n} \quad & f(x) \\
    \mathrm{s.t.} \quad & c_k(x) \geq 0, \quad k = 0,\ldots, N,
\end{split}
\]
where $\{c_k: \reals^n \rightarrow \reals^{m_k}, k=0,\ldots, N\}$ is a set of vector-valued inequality constraints, with corresponding Jacobians $J_k(x) := \partial c_k(x)/\partial x$. Let $\bm{y} = (y_0,\ldots, y_N)$ be the stacked dual vector for the inequality constraints, $H(x, \bm{y})$ the Lagrangian Hessian, and $g(x)$ the objective gradient. The QP sub-problem is then defined as:
\[
    \begin{split}
        \delta x^* = \argmin_{\delta x \in \reals^n} \quad & \psi(\delta x) := g(x)^T \delta x + \dfrac{1}{2} \delta x^T H(x, \bm{y}) \delta x \\ \mathrm{s.t.} \quad & c_k(x) + J_k(x) \delta x \geq 0, \quad k = 0,\ldots, N.
    \end{split}
\]
Let $\hat{y}_k$ be the optimal dual vector for the QP-subproblem for constraint $k$ and define $\delta y_k^* := \hat{y}_k - y_k$. Consider the Augmented Lagrangian, $\mathcal{M}$, defined as:
\[
    \mathcal{M}(x, \bm{y}, \bm{s}; \{\rho_k\}_{k=0}^{N}) = f(x) - \sum_{k=0}^{N} y_k^T (c_k(x) - s_k) + \sum_{k=0}^{N} \dfrac{1}{2} \rho_k \|c_k(x) - s_k\|^2 \quad (s_k \geq 0),
\]
where $s_k \in \reals^{m_k}_{\geq 0}$ is the ``slack" for the $k^\mathrm{th}$ constraint, and $\bm{s}$ denotes the concatenatation $(s_0, \ldots, s_N)$. At the current iterate $(x, \bm{y})$ set the slack vectors $s_k$ as follows:
\[
    s_k = \begin{cases} \max(0, c_k) \quad &\text{if} \ \rho_k = 0 \\
        \max(0, c_k - y_k/\rho) \quad &\text{otherwise}
        \end{cases},
\]
and set the slack search directions $\delta s_k^* = c_k(x) + J_k(x) \delta x^* - s_k$. Define the line-search function $\phi(\alpha; \rho)$ as:
\[
    \phi(\alpha; \rho) := \mathcal{M}(x + \alpha \delta x^*, \bm{y} + \alpha \bm{\delta y}, \bm{s} + \alpha \bm{\delta s}; \rho).
\]
Now, we wish to choose the set of penalty parameters $\{\rho_k^+\}_{k=0}^{N}$ s.t. $\phi'(0; \rho^+) \leq -\dfrac{1}{2} \delta x^{*^T} H(x, \bm{y}) \delta x^* := -\dfrac{1}{2} \Delta^*$. Consider the gradient $\phi'(0; \rho^+)$:
\[
\begin{split}
    \phi'(0; \rho^+) &= g^T \delta x^* + \sum_{k=0}^{N}\left[ -\delta x^{*^T} J_k^T y_k + \rho_k^+ \delta x^{*^T} J_k^T (c_k - s_k) - (c_k - s_k)^T \delta y_k^* + y_k^T \delta s_k^* - \rho_k^+ \delta s_k^{*^T}(c_k - s_k) \right] \\
    &= g^T \delta x^* + \sum_{k=0}^{N} (2y_k - \hat{y}_k)^T (c_k - s_k) - \rho_k^+ \|c_k - s_k\|^2,
\end{split}
\]
where we have used the definitions of $\delta s_k^*$ and $\delta y_k^*$. Notice now that if $c_k - s_k = 0$ for all $k$, then since $s_k \geq 0$, $\delta x = 0$ is a feasible solution to the QP sub-problem. Thus, $\psi(\delta x^*) \leq \psi(0)$, from which one obtains the desired inequality. Consider then the case that the set $\mathcal{I} := \{k : \|c_k - s_k\| > 0\}$ is non-empty. Then, we can re-write $\phi'(0; \rho^+)$ as:
\[
    \sum_{k \in \mathcal{I}} \dfrac{1}{|\mathcal{I}|} g^T \delta x^* + (2y_k - \hat{y}_k)^T (c_k - s_k) - \rho_k^+ \|c_k - s_k\|^2
\]
Now, for $\rho_k^+ = \hat{\rho}_k, k \in \mathcal{I}$, defined as:
\[
    \hat{\rho}_k := \dfrac{\frac{1}{|\mathcal{I}|}\psi(\delta x^*) + (2y_k - \hat{y}_k)^T (c_k - s_k)}{\|c_k - s_k\|^2},
\]
we obtain:
\[
    \phi'\left(0; \{\hat{\rho}_k\}_{k\in \mathcal{I}} \cup \{\rho_k\}_{k \notin \mathcal{I}}\right) = -\dfrac{1}{2} \Delta^*.
\]
Thus, the update equation for $\{\rho_k\}$ may be written as:
\[
    \rho_k^+ = \begin{cases} \max(2 \rho_k, \hat{\rho}_k) \quad &\text{if} \quad k \in \mathcal{I} \\
    \rho_k \quad &\text{otherwise}.
    \end{cases}
\]
In context of Section~\ref{sec:line-implicit}, the term $\psi(\delta x^*)$ is equivalent to the optimal objective of problem~\eqref{qp_prob_implicit}, while the correspondences for the dual and slack vectors follows straightforwardly. The line-search acceptance conditions are as given in~\eqref{ls_condn}.

\subsection{Termination Conditions}
Let $\tau_p, \tau_d \in \reals_{>0}$ be user-specified relative tolerances, and define $\tau_x = \tau_p(1 + \|\bm{u}\|)$ and $\tau_y = \tau_d(1 + \|\bm{y}\|)$. Then, the KKT stationarity termination conditions are given as:
\begin{equation}
    \begin{split}
    &k = 0,\ldots, N: \begin{cases}
    c_k \geq -\tau_x \\
    y_k \geq -\tau_y \\
    \|c_k \circ y_k \|_{\infty} \leq \tau_y
    \end{cases} \\
    &k = 0,\ldots, N-1:  \|\nabla_{u_k} {\hat{H}}_k\|_{\infty} \leq \tau_y
    \end{split}
\label{shooting-term}    
\end{equation}
where $\circ$ denotes the Hadamard product, and $\hat{H}_k$, the Hamiltonian, is defined in~\eqref{qp_obj_implicit_hess}. 

\section{Computing Sensitivity Gains using DP}
\label{app:proof}

\begin{proof}[Proof of Theorem~\ref{thm:sens-gains}]
We will prove the result via induction. Notice that for the base case $k=N-1$, the cost-to-go function $\delta V_N(\delta x_N)$ is defined only for $\delta x_N \in \X_N$, which by construction, is the set $\CR_{N}$. Thus $\delta x_N^* \in \CR_N$ and $\delta \tilde{V}_N(\delta x_N) = \delta V_N(\delta x_N)$ for $\delta x_N \in \CR_{N}$. It follows that problem~\eqref{prob_k} for $k=N-1$ coincides exactly with the definition for $\delta \pi_{N-1}^*$ given in~\eqref{pi_pert}. Thus, $\delta \hat{\pi}^*_{N-1}(\delta x_{N-1}) = \delta \pi^*_{N-1}(\delta x_{N-1})$ for all $\delta x_{N-1} \in \delta X_{N-1}$ where the problem is feasible.

Now, notice that problem~\eqref{prob_k} for $k=N-1$ is a multi-parametric QP in the ``parameter" $\delta x_{N-1}$. Thus, by Theorem 2 in~\cite{BemporadMorari2002}, linear independence of the active inequality constraints at $\delta \hat{\pi}^*_{N-1}(\delta x_{N-1}^*)$ implies that the solution functions $\delta \hat{\pi}^*_{N-1}(\delta x_{N-1})$ and $\hat{y}^*_{N-1}(\delta x_{N-1})$ are locally\footnote{One can actually show that the solution function $\delta \hat{\pi}^*_{N-1}$ is in fact piecewise affine over $\delta \X_{N-1}$ however this is not necessary for the proof.} affine in a region containing $\delta x_{N-1}^*$, taking the form in~\eqref{sens_aff}. This region containing $\delta x_{N-1}^*$ is precisely the set:
\[
    \{\delta x_{N-1} \in \delta \X_{N-1}: \delta \hat{\pi}^*_{N-1,a}(\delta x_{N-1}) = \delta \hat{\pi}^*_{N-1}(\delta x_{N-1})\}.
\]
A second consequence of the aforementioned theorem is that the above set is defined by the intersection of the following two sets:
\[
\begin{split}
    &\{\delta x_{N-1} \in \delta \X_{N-1} : \bar{G}_{N-1} \begin{bmatrix} \delta \hat{\pi}^*_{N-1,a}(\delta x_{N-1}) \\ A_{N-1}\delta x_{N-1} + B_{N-1}\hat{\pi}^*_{N-1,a}(\delta x_{N-1}) \end{bmatrix} \leq \bar{h}_{N-1}\} \\
    &\{\delta x_{N-1} \in \delta \X_{N-1} : \hat{y}^*_{N-1,a}(\delta x_{N-1}) \geq 0\},
\end{split}
\]
that is, the subset of $\delta \X_{N-1}$ where the locally affine solutions satisfy the primal-dual feasibility constraints of problem~\eqref{prob_k} for $k=N-1$. This however is precisely the definition of the set $\CR_{N-1}$ in~\eqref{ctg_cons}. Thus, we have established that $\delta \hat{\pi}^*_{N-1,a}(\delta x_{N-1}) = \delta \pi^*_{N-1}(\delta x_{N-1})$ for $\delta x_{N-1} \in \CR_{N-1}$ and by construction, $\delta x_{N-1}^* \in \CR_{N-1}$. Substituting the affine feedback law into the objective for problem~\eqref{prob_k} for $k=N-1$ gives the recursion for $\{P_{N-1}, p_{N-1}, v_{N-1}\}$, as defined in~\eqref{ctg_recurse}. Thus, $\delta \tilde{V}_{N-1}(\delta x_{N-1}) = \delta V_{N-1}(\delta x_{N-1})$ for $\delta x_{N-1} \in \CR_{N-1}$, completing the proof for $k=N-1$.

Suppose then that the following statements are true for some $k+1 \leq N-1$: (i) $\delta x_{k+1}^* \in \CR_{k+1}$, and (ii) $\delta V_{k+1}(\delta x_{k+1}) = \delta \tilde{V}_{k+1}(\delta x_{k+1})$ for $\delta x_{k+1} \in \CR_{k+1}$. Consider the definition of $\delta \pi^*_k$ given in~\eqref{pi_pert}, and define:
\[
\mathrm{Pre}^*(\CR_{k+1}) := \{\delta x_{k} \in \delta \X_{k} : A_k \delta x_{k} + B_{k} \delta \pi^*_{k}(\delta x_{k}) \in \CR_{k+1}\}
\]
Since $\delta x_{k+1}^* \in \CR_{k+1}$ by the inductive hypothesis, it follows that $\delta x_{k}^* \in \mathrm{Pre}^*(\CR_{k+1})$ and hence, the set is non-empty. Now, for all $\delta x_{k} \in \mathrm{Pre}^*(\CR_{k+1})$, one may add the redundant constraint: $\delta x_{k+1} \in \CR_{k+1}$ and equivalently re-write $\delta \pi_{k}^*$ as:
\[
\begin{split}
    \argmin_{\delta u_{k} \in \delta \U_k} \quad &\tilde{l}_{k}(\delta x_{k}, \delta u_{k}) + \delta V_{k+1}(\delta x_{k+1}) \\
    \mathrm{s.t.} \quad & \delta x_{k+1} \in \CR_{k+1}.
\end{split}
\]
Now, leveraging the inductive hypothesis that $\delta \tilde{V}_{k+1}(\delta x_{k+1}) = \delta V_{k+1}(\delta x_{k+1})$ for all $\delta x_{k+1} \in \CR_{k+1}$, we can re-write $\delta \pi_{k}^*$ as:
\[
\begin{split}
    \argmin_{\delta u_{k} \in \delta \U_k} \quad &\tilde{l}_{k}(\delta x_{k}, \delta u_{k}) + \delta \tilde{V}_{k+1}(\delta x_{k+1}) \\
    \mathrm{s.t.} \quad & \delta x_{k+1} \in \CR_{k+1}.
\end{split}
\]
Comparing with problem~\eqref{prob_k}, this is precisely the definition of $\delta \hat{\pi}^*_{k}$. Thus, we have that $\delta \hat{\pi}^*_{k}(\delta x_{k}) = \delta \pi^*_{k}(\delta x_{k})$ for $\delta x_{k} \in \mathrm{Pre}^*(\CR_{k+1})$.

Now, consider the fixed affine feedback law $\delta \hat{\pi}^*_{k,a}$ in~\eqref{sens_aff}, that is well-defined thanks to the linear independence assumption at $\delta x_k^*$. Once again, applying Theorem 2 in~\cite{BemporadMorari2002}, we have that $\delta \hat{\pi}^*_k(\delta x_k) = \delta \hat{\pi}^*_{k,a}(\delta x_k)$ for $\delta x_k \in \CR_k$, as defined in~\eqref{ctg_cons}. Substituting the affine law into the objective of problem~\eqref{prob_k} yields the recursion in~\eqref{ctg_recurse}; thus, $\delta \tilde{V}_k(\delta x_k) = \delta V_k(\delta x_k)$ for all $\delta x_k \in \CR_k \cap \mathrm{Pre}^*(\CR_{k+1})$. Further, by construction, $\delta x_k^* \in \CR_k$; thus, the intersection $\CR_k \cap \mathrm{Pre}^*(\CR_{k+1})$ is non-empty.

To finish the proof, we must show that $\CR_{k} \subseteq \mathrm{Pre}^*(\CR_{k+1})$. By doing so, we establish the equivalency between the fixed affine feedback law $ \delta \hat{\pi}^*_{k,a}$ and $\delta \pi^*_{k}$ and consequently, the equivalency of $\delta V_{k}$ and $\delta \tilde{V}_{k}$ for $\delta x_{k} \in \CR_{k}$.

To show this last part, notice that for any $\delta x_k$ that lies on the boundary $\partial \mathrm{Pre}^*(\CR_{k+1})$, it holds that $A_k \delta x_k + B_k \delta \pi_k^*(\delta x_k)$ lies on the boundary $\partial \CR_{k+1}$. This follows from the continuity of $\delta\pi_k^*(\delta x_k)$. Then, by the equivalence of $\delta \pi^*_k$ and $\delta \hat{\pi}_k^*$ for $\delta x_k \in \mathrm{Pre}^*(\CR_{k+1})$ and the equivalence of $\delta \hat{\pi}^*_k$ and $\delta \hat{\pi}^*_{k,a}$ for $\delta x_k \in \CR_k$, for any $\delta x_k \in \CR_k \cap \partial \mathrm{Pre}^*(\CR_{k+1})$, it holds that $A_k \delta x_k + B_k \delta\hat{\pi}_{k,a}^*(\delta x_k) \in \partial \CR_{k+1}$. It follows that $\delta x_k \in \partial \CR_k$.

Thus, since the intersection $\CR_k \cap \mathrm{Pre}^*(\CR_{k+1})$ is non-empty, then either (i) $\CR_k \subset \mathrm{Pre}^*(\CR_{k+1})$, or (ii) $\CR_k \cap \partial \mathrm{Pre}^*(\CR_{k+1})$ is non-empty. In the latter case, we have shown that is the set $\partial \CR_k \cap \partial \mathrm{Pre}^*(\CR_{k+1})$, completing the proof.

\end{proof}

\section{Experiment Details}
\label{app:numerics}

The following hyperparameters were held fixed for all Shooting SQP variations over all environments:
\begin{table}[H]
\centering
\begin{tabular}{l | c}
\toprule
Parameter & Value \\
\midrule
Line-search decrease ratio: $\sigma$ & 0.4 \\ 
Line-search gradient ratio: $\eta$ & 0.49 \\
Line-search step-size lower-bound: $\underline{\alpha}$ & 1e-5 \\
Termination Primal tolerance: $\tau_p$ & 1e-3 \\ 
Termination Dual tolerance: $\tau_d$ & 1e-3* \\
\bottomrule
\end{tabular}
\caption{Hyperparameters for Shooting SQP. (*) For the Quad-Pendulum example, the dual tolerance $\tau_d$ was set as $10^{-2}$.}
\label{tab:sqp_hyper}
\end{table}
In the following, we give details specific to each environment case-study.

\subsection{Motion Planning for a Car}
This is a system with 4 states: $x = (p_x, p_y, \theta, v)$, where $(p_x, p_y) \in \reals^2$ is the 2D position, $\theta \in \mathcal{S}^1$ is the orientation, and $v \in \reals$ is the forward velocity. The continuous-time dynamics are given by:
\[
    \dot{p}_x = v \sin \theta, \quad \dot{p}_y = v \cos \theta, \quad \dot{\theta} = v u^{\theta}, \quad \dot{v} = u^v,
\]
where $(u^{\theta}, u^v)$ are the steering and acceleration inputs. As in~\cite{XieLiu2017}, we use Euler integration with a time-step of $0.05$s to yield the discrete-time model. The control limits are taken from~\cite{AoyamaBoutselis2020}: $u^{\theta} \in [-\frac{\pi}{3}, \frac{\pi}{3}]$ rad/s, and $u^{v} \in [-6, 6]\ \mathrm{s}^{-2}$.

The stage cost is $l_k(x, u) = 0.05 u^T R u$ with $R = \mathrm{diag}(0.2, 0.1)$, and terminal cost is $l_N(x) = (x - x_g)^T Q_N (x - x_g)$ with $Q_N = \mathrm{diag}(50, 50, 50, 10)$ and $x_g = (3, 3, \frac{\pi}{2}, 0)$. The horizon $N$ is 40 steps. The initial control sequence guess was all zeros. The constant $\gamma$ for $\clg$ was set to $10^{-4}$.

Figure~\ref{fig:car-steps} shows the solver progress for all algorithms and cases.

\begin{figure}[H]
  \centering
  \includegraphics[width=0.5\textwidth]{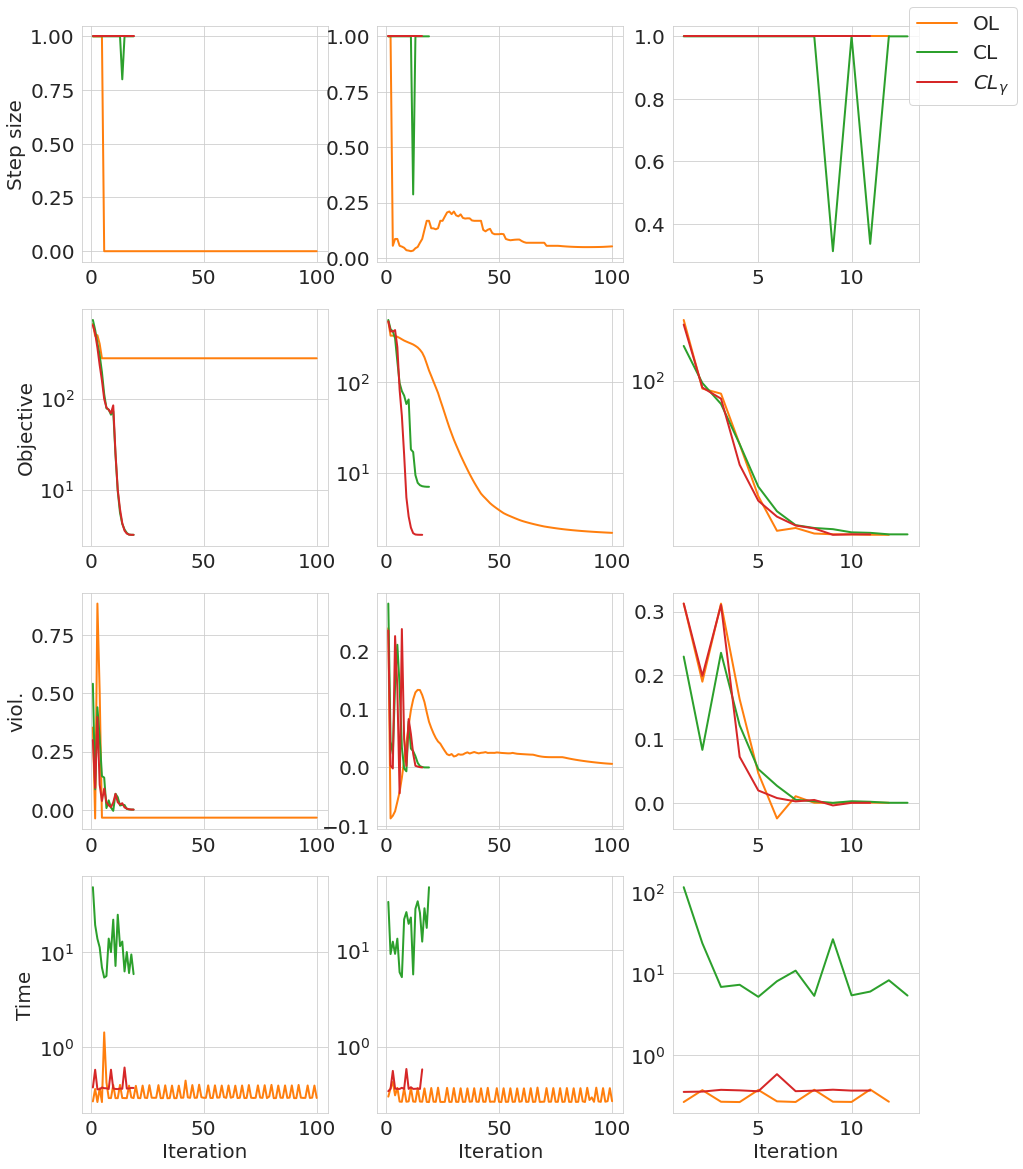}  
\caption{Step-size, objective, constraint violation (positive indicates strict feasibility), and computation time per iteration, as a function of iteration. \emph{Left-to-right}: Case index.}
\label{fig:car-steps}
\end{figure}

\subsection{Acrobot}

We study the classic acrobot ``swing-up" task with $n=4$ states and $m=1$ control input. The state is defined as $x = (q_1, q_2, v_1, v_2)$ where $q_1$ the root joint angle, $q_2$ is the relative elbow joint angle, and $v_1, v_2$ are their respective angular rates. The dynamics were taken from~\cite{Tedrake2009}. Starting from an initial state $x_0 = 0$, we require the acrobot to swing up to the goal state $x_g = (\pi, 0, 0, 0)$, subject to the control limits $u \in [-2, 2]$. The stage cost is $l_k(x, u) = \frac{1}{2}\left[w_1(\cos q_1 + \cos(q_1 + q_2) + 2) + w_2 u^2\right]$, and the terminal cost is $l_N(x) = \frac{1}{2} w_3 \|x - x_g\|^2$, with weights $(w_1, w_2, w_3) = (0.1, 0.01, 10)$. Additionally, we enforce the limit constraints\footnote{Due to the fact that these two state components really live on $\mathcal{S}^1 \times \mathcal{S}^1$, we employ chart switching to smoothly handle this constraint near the boundary $-\pi/\pi$. This is achieved by adjusting the boundary constraints for each angle to lie in $[-\pi \pm \delta_k, \pi \pm \delta_k]$ for the next iterate, where the per time-step adjustment $\delta_k$ is chosen based on the current values for $q_1$ and $q_2$ at time-step $k$. We also wrap both angles back to the interval $[-\pi, \pi]$ at the end of each iteration.} $(q_1, q_2) \in [-\pi, \pi]^2$ and the terminal constraint: $\|x_N - x_g\|^2 \leq 0.2^2$. The discrete-time dynamics were obtained using explicit Euler integration with a time-step of 0.05s, and the horizon $N$ was $150$ steps.

To generate the initial guess, we set $\bm{\chi}_0$, the initial guess for the state-trajectory to be a linear interpolation between $x_0$ and $x_g$, and $\bm{\mu}_0$, the initial guess for the controls as all zeros. Then, the following rollout was used to generate the initial control trajectory for SQP:
\[
    x_{k+1} = f(x_k, \mathrm{clip}_{\ulb}^{\uub} [\mu_k + K_k^{\mathrm{lqr}} (x_k - \chi_k)]),\ x_0 = 0,\ k=0,\ldots, N-1,
\]
where the gains $\{K_k^{\mathrm{lqr}}\}$ were computed via a TV-LQR solve, with cost defined by the Hessian of the objective and dynamics linearized about $(\bm{\chi}_0, \bm{\mu}_0)$.

We found that for this problem, the \emph{Gauss-Newton} Hessian approximation was more stable for both algorithms. This entails dropping the second-order gradients stemming from $\hat{\nu}_{k+1}^T f(x_k, u_k)$ in~\eqref{qp_obj_implicit_hess}. The constant $\gamma$ for $\clg$ was set to $10^{-4}$. Figure~\ref{fig:acro-steps} illustrates solver progress.
\begin{figure}[H]
  \centering
  \resizebox{0.5\textwidth}{!}{\includegraphics[width=.5\textwidth]{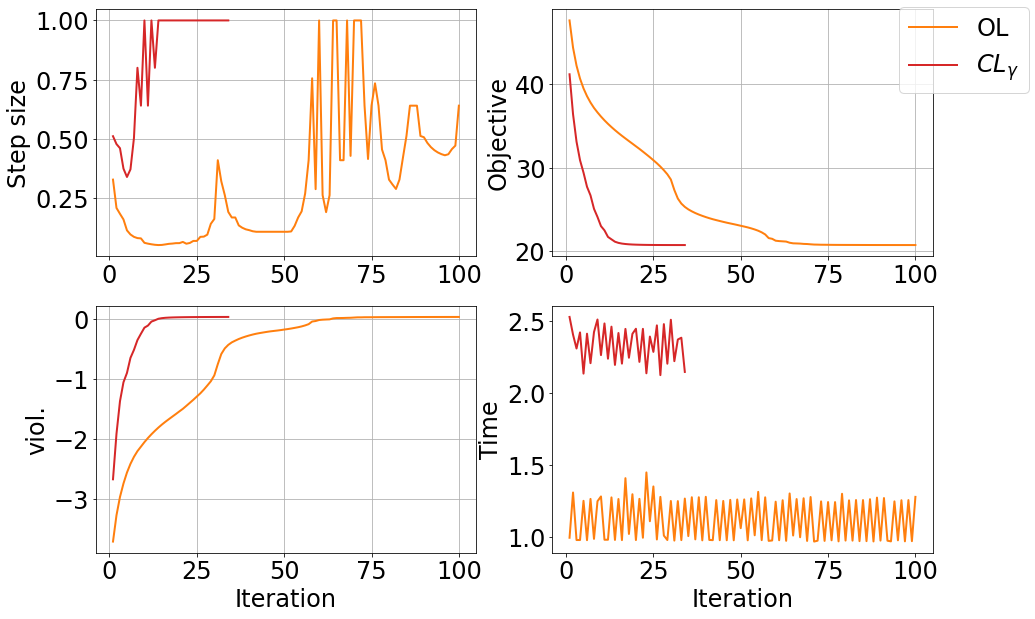}}
\caption{Step-size, objective, constraint violation (positive indicates strict feasibility), and computation time per iteration.}
\label{fig:acro-steps}
\end{figure}
Despite sub-optimal step-sizes for both algorithms (due to the natural instabilities of the dynamics), $\clg$ successfully converges within 33 iterations while $\ol$ is still struggling to resolve the Lagrangian stationarity residual at 100 iterations.

\subsection{Quad-Pendulum}

The state is defined as $x = (p_x, p_z, \theta, \phi, \dot{p}_x, \dot{p}_z, \dot{\theta}, \dot{\phi})$, where $(p_x, p_z) \in \reals^2$ is the CoM position of the quadrotor in the vertical plane, $\theta \in \mathcal{S}^1$ is its orientation, and $\phi \in \mathcal{S}^1$ is the orientation of the attached pendulum w.r.t. the vertical. The two control inputs $u = (u_1, u_2)$ corresponding to the rotor thrusts are subject to the limits $[0.1, 3] m_q g$, where $m_q$ is the mass of the quadrotor and $g$ is the gravitational acceleration.

We provide the Lagrangian characterization of the dynamics. The mass matrix $M(q)$, as a function of the generalized coordinates $q = (p_x, p_z, \theta, \phi)$ is given by:
\[
    \begin{bmatrix} m_q + m_p & 0 & 0 & m_p L \cos\phi \\
                    0 & m_q + m_p & 0 & m_p L \sin\phi \\
                    0 & 0 & J & 0 \\
                    m_p L \cos\phi & m_p L \sin\phi & 0 & m_p L^2
    \end{bmatrix},
\]
where $m_p$ is the mass of the pendulum (assumed concentrated at the endpoint), $L$ is the pendulum length, and $J$ is the quadrotor moment of inertia. The kinetic energy is thus: $T(q, \dot{q}) = \frac{1}{2} \dot{q}^T M(q) \dot{q}$; the potential energy is $V(q) = m_q g p_z + m_p g (p_z - L \cos\phi)$, and the mechanical Lagrangian is $\mathcal{L}(q, \dot{q}) = T(q, \dot{q}) - V(q)$. The dynamics are thus given as:
\[
    \dfrac{d}{dt}\left(\dfrac{\partial \mathcal{L}(q, \dot{q})}{\partial \dot{q}}\right) - \dfrac{\partial \mathcal{L}(q, \dot{q})}{\partial q} = F(q),
\]
where $F(q)$, the generalized force vector is given as:
\[
    F(q) = \begin{bmatrix} -(u_1 + u_2) \sin \theta \\ (u_1 + u_2) \cos \theta \\ (u_1 - u_2)l - \tau_f \\ \tau_f \end{bmatrix},
\]
where $2l$ is the quadrotor wing-span, $\tau_f$ is the viscous frictional torque: $-\nu(\dot{\phi} - \dot{\theta})$, where $\nu$ is the constant friction coefficient. We used the constants: $m_q = 0.486$, $m_p = 0.2m_q$, $l=0.25$, $L=2l$, $g=9.81$, $J=0.00383$, $\nu=0.01$, and employed Euler discretization with time-step 0.025 s.

The stage cost is $l_k(x, u) = \frac{1}{2} (w_1 (\|(p_x, p_z, \theta) - (g_x, g_z, g_\theta)\|^2 + (1 + \cos\phi)) + w_2 \|u - u_h\|^2)$, and the terminal cost is $l_N(x) = \frac{1}{2} (x - x_g)^T Q_N (x - x_g)$, where $x_g = (g_x, g_z, g_{\theta}, g_{\phi}, 0, 0, 0, 0)$ is the goal state with: $(g_x, g_z, g_\theta, g_\phi) = (3, -1.5, 0, \pi)$. The vector $u_h = 0.5(m_q + m_p)g \bm{1}$ is the hover thrust setpoint, and $Q_N = \mathrm{diag}(10, 10, 1, 1, 1, 1, 1, 1)$. The weights are $(w_1, w_2, w_3) = (0.01, 0.05, 5)$. The horizon $N$ was 160 steps.

In addition to the obstacle avoidance constraints, we also set the operational constraints $(p_x, p_z) \in [-4, 4] \times [-2, 2]$ and $\theta \in [-\frac{3\pi}{4}, \frac{3\pi}{4}]$. The initial control sequence guess was set to $u_h$, the hover setpoint, for all timesteps.

Due to the difficulty of this problem, the constant $\gamma$ for $\clg$ was initialized at $10^{-3}$ and reduced by a factor of 10 every SQP iteration, until a lower-bound of $10^{-5}$.

Figure~\ref{fig:quadpole1} shows a timelapse for the first case solution from $\clg$, and in Figure~\ref{fig:quadpole_err} we plot the re-construction error $\|\delta \hat{\pi}_k^*(\delta x_k^*) - \delta u_k^*\|$ for both cases for $\clg$.

\begin{figure}[H]
  \centering
  \includegraphics[width=0.7\textwidth]{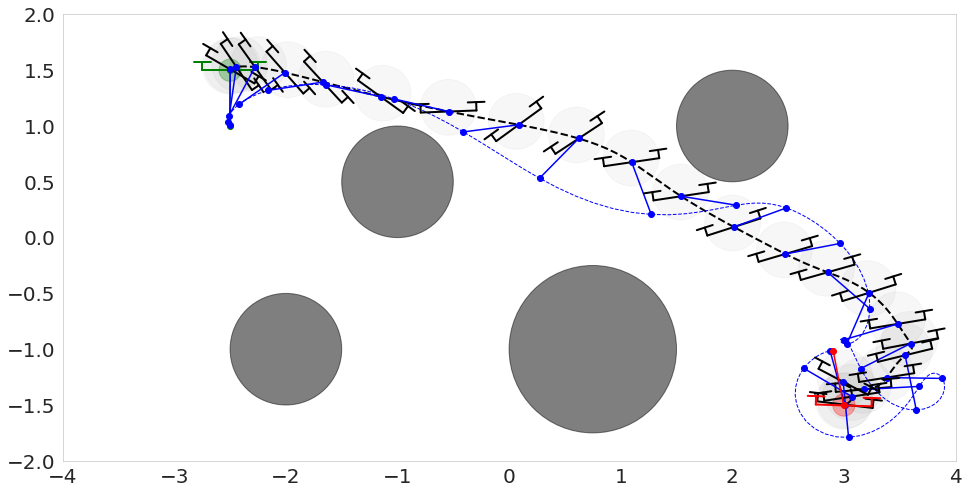}  
\caption{Timelapse of the solution obtained by $\clg$ for Case \#1. Robot snapshots are 0.1s apart at the beginning and expand to 0.3s towards the end; initial render is in green, final render is in red. Light shaded circle centered on the quadrotor is the circumscribing disc used for the collision-avoidance constraint for the quadrotor body (in addition to the avoidance constraint w.r.t. the pendulum pole).}
\label{fig:quadpole1}
\end{figure}

\begin{figure}[H]
  \centering
  \includegraphics[width=0.7\textwidth]{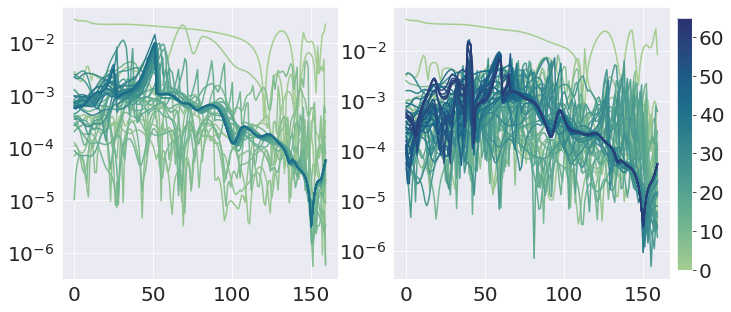}  
\caption{Re-construction error between open-loop solution $\bm{\delta u}^*$ and perturbation policy solution $\delta \hat{\pi}_k^*(\delta x_k^*)$ as a function of time-step across all SQP iterations. \emph{Left}: Case \#1; \emph{Right}: Case \#2. Colorbar identifies the SQP iteration.}
\label{fig:quadpole_err}
\end{figure}

\end{document}